\documentclass[11pt]{amsart}

\usepackage{fancybox,color}
\usepackage{latexsym,amsmath,amssymb,amsbsy,amscd,bm,amsfonts}
\usepackage{marginnote}
\usepackage[active]{srcltx}
\usepackage{hyperref}
\hypersetup{
  colorlinks   = true, 
  urlcolor     = blue, 
  linkcolor    = blue, 
  citecolor    = red   
}
\usepackage[utf8]{inputenc}
\usepackage[T1]{fontenc}

\newlength{\hchng}
\newlength{\vchng}
\def\1{\raisebox{2pt}{\rm{$\chi$}}}

\newtheorem{theorem}{Theorem}
\newtheorem{corollary}[theorem]{Corollary}
\newtheorem{lemma}[theorem]{Lemma}

\newtheorem{definition}[theorem]{Definition}
\newtheorem{remark}[theorem]{Remark}

\newcommand{\R}{{\mathbb R}}
\newcommand{\E}{{\mathbb E}}
\newcommand{\PP}{{\mathbb P}}

 \newcommand{\eps}{{\varepsilon}}
 \def\1{\raisebox{2pt}{\rm{$\chi$}}}

\usepackage{enumerate}

\def\vint_#1{\mathchoice%
          {\mathop{\kern 0.2em\vrule width 0.6em height 0.69678ex depth -0.58065ex
                  \kern -0.8em \intop}\nolimits_{\kern -0.4em#1}}%
          {\mathop{\kern 0.1em\vrule width 0.5em height 0.69678ex depth -0.60387ex
                  \kern -0.6em \intop}\nolimits_{#1}}%
          {\mathop{\kern 0.1em\vrule width 0.5em height 0.69678ex
              depth -0.60387ex
                  \kern -0.6em \intop}\nolimits_{#1}}%
          {\mathop{\kern 0.1em\vrule width 0.5em height 0.69678ex depth -0.60387ex
                  \kern -0.6em \intop}\nolimits_{#1}}}
\def\vintslides_#1{\mathchoice%
          {\mathop{\kern 0.1em\vrule width 0.5em height 0.697ex depth -0.581ex
                  \kern -0.6em \intop}\nolimits_{\kern -0.4em#1}}%
          {\mathop{\kern 0.1em\vrule width 0.3em height 0.697ex depth -0.604ex
                  \kern -0.4em \intop}\nolimits_{#1}}%
          {\mathop{\kern 0.1em\vrule width 0.3em height 0.697ex depth -0.604ex
                  \kern -0.4em \intop}\nolimits_{#1}}%
          {\mathop{\kern 0.1em\vrule width 0.3em height 0.697ex depth -0.604ex
                  \kern -0.4em \intop}\nolimits_{#1}}}

\newcommand{\aveint}[2]{\mathchoice%
          {\mathop{\kern 0.2em\vrule width 0.6em height 0.69678ex depth -0.58065ex
                  \kern -0.8em \intop}\nolimits_{\kern -0.45em#1}^{#2}}%
          {\mathop{\kern 0.1em\vrule width 0.5em height 0.69678ex depth -0.60387ex
                  \kern -0.6em \intop}\nolimits_{#1}^{#2}}%
          {\mathop{\kern 0.1em\vrule width 0.5em height 0.69678ex depth -0.60387ex
                  \kern -0.6em \intop}\nolimits_{#1}^{#2}}%
          {\mathop{\kern 0.1em\vrule width 0.5em height 0.69678ex depth -0.60387ex
                  \kern -0.6em \intop}\nolimits_{#1}^{#2}}}

\newcommand{\ud}{\, d}

\newcommand{\abs}[1]{\left| #1 \right|}

\newcommand{\ol}{\overline}
\newcommand{\dist}{\operatorname{dist}}
\newcommand{\I}{\textrm{I}}

\begin{document}

\title[Games for Pucci's maximal operators]
{\bf Games for Pucci's maximal operators}

\author[P. Blanc, J. J. Manfredi and J. D. Rossi]{Pablo Blanc, Juan J. Manfredi,  and Julio D. Rossi}

\date{}

\begin{abstract} In this paper we introduce a game whose value functions converge
(as a parameter that measures the size of the steps goes to zero) uniformly to solutions to
the second order Pucci maximal operators.
\end{abstract}

\maketitle


\section{Introduction}

Our main goal in this paper is to describe a game whose
values approximate viscosity solutions to the maximal Pucci problem
\begin{equation}\label{1.1}
\left\{
\begin{array}{ll}
\displaystyle
P_{\lambda, \Lambda}^+ (D^2 u) := \Lambda \sum_{\lambda_j>0} \lambda_j
+\lambda \sum_{\lambda_j<0} \lambda_j  = f , \qquad & \mbox{ in } \Omega, \\[10pt]
u=g , \qquad & \mbox{ on } \partial \Omega.
\end{array}
\right.
\end{equation}
Here $D^2 u$ is the Hessian matrix and  $\lambda_i(D^2 u)$ denote its eigenvalues.
The function $f$ is assumed to be uniformly continuous.
We assume that the ellipticity constants verify $\Lambda>\lambda>0$.

Let us describe the game that we propose to approximate solutions to \eqref{1.1}. 
It is a single player game (that tries to maximize the expected outcome). 
It  can also  be viewed as a optimal control problem.
Fix a bounded domain $\Omega \subset \mathbb{R}^N$ that satisfies a \textit{uniform exterior sphere condition}. Fix 
a \textit{running payoff function} $f :\Omega \mapsto \mathbb{R}$ 
and a \textit{final payoff function} $g :\mathbb{R}^N \setminus \Omega \mapsto \mathbb{R}$. \par
The rules of the game are as follows:
a token is placed  at an initial position $x_0 \in \Omega$, the player
chooses an orthonormal basis of ${\mathbb{R}^N}$, $v_1,...,v_N$ and then,  for each $v_i$, he
chooses either $\mu_i = \sqrt{\lambda}$ or $\mu_i = \sqrt{\Lambda}$. Then the position of the token
is moved to $x\pm \eps  \mu_i v_i$ with equal probabilities $\frac{1}{2N}$.
The game continues until the position of the token leaves the domain and at this 
point $x_\tau$ the payoff is given by $g(x_\tau)-\frac{1}{2N}\eps^2\sum_{k=0}^{\tau-1}f(x_k).$
For a given strategy $S_I$ (the player choses an orthonormal basis and the set
of corresponding $\mu_i$ at every step of the game) 
we  compute the expected outcome as
$$
\mathbb{E}_{S_I}^{x_0} \left[g(x_\tau)-\frac{1}{2N}\eps^2\sum_{k=0}^{\tau-1}f(x_k)\right].
$$
Then the value of the game for any $x_0 \in \Omega$ 
is defined as
$$
u^\eps (x_0) := \sup_{S_{I}} \mathbb{E}_{S_I}^{x_0} \left[g(x_\tau)-\frac{1}{2N}\eps^2\sum_{k=0}^{\tau-1}f(x_k)\right].
$$

Our first result states that the value of this game verifies a Dynamic Programming Principle.

\begin{theorem} \label{teo.dpp}
The value of the game 
$
u^\eps$ verifies 
\begin{equation}\label{eq.DPP}
\left\{
\begin{array}{ll}
\displaystyle u^\eps (x) =
-\frac{1}{2N}\eps^2f(x)
+ \frac{1}{2N} \sup_{v_i,\mu_i}
\sum_{i=1}^N u^\eps (x + \eps \mu_i v_i) + u^\eps (x - \eps \mu_i v_i)
& x \in \Omega, \\[10pt]
u^\eps (x) = g(x)  & x \not\in \Omega.
\end{array}
\right.
\end{equation}
\end{theorem}

Our next goal is to look for the limit as $\eps \to 0$. 

\begin{theorem} \label{teo.converge}
Let $u^\eps$ be the values of the game. Then,
\begin{equation}\label{eq.converge}
u^\eps \to u, \qquad \mbox{ as } \eps \to 0,
\end{equation}
uniformly in $\overline{\Omega}$. The limit $u$ is the unique viscosity
solution to
\begin{equation}\label{1.1.teo}
\left\{
\begin{array}{ll}
\displaystyle
P_{\lambda, \Lambda}^+ (D^2 u) := \Lambda \sum_{\lambda_j>0} \lambda_j
+\lambda \sum_{\lambda_j<0} \lambda_j  = f , \qquad & \mbox{ in } \Omega, \\[10pt]
u=g , \qquad & \mbox{ on } \partial \Omega.
\end{array}
\right.
\end{equation}
\end{theorem}

\begin{remark} {\rm When $f$ is assumed to be positive,
one can play the game described before, but with the following variant: the player 
chooses $v_{1},...,v_k$ orthonormal vectors (notice that the number of vectors $k$ is part of the Player's choice).
Then, the new position of the game goes to $x\pm \eps  \sqrt{\Lambda} v_i$   with equal probabilities $1/2N$
or remain fixed at $x$ with probability $1- k/N$.
(note that there is no choice of $\mu_i$ among $\lambda$, $\Lambda$).
In this case the value of the game 
$
u^\eps$ verifies 
$$
\left\{
\begin{array}{ll}
\displaystyle u^\eps (x) =
-\frac{1}{2N}\eps^2f(x)
+ \frac{1}{2N} \sup_{v_i}
\sum_{i=1}^k u^\eps (x + \eps \sqrt{\Lambda} v_i)  + u^\eps (x - \eps \sqrt{\Lambda} v_i)
+ \frac{N-k}{N} u^\eps (x)
& x \in \Omega, \\[10pt]
u^\eps (x) = g(x)  & x \not\in \Omega.
\end{array}
\right.
$$
Now, with the same ideas used to deal with Pucci's maximal operator, one can pass to the limit as $\eps \to 0$ and find a 
viscosity solution to the degenerate problem
$$
\left\{
\begin{array}{ll}
\displaystyle
P_{0, \Lambda}^+ (D^2 u) := \Lambda \sum_{\lambda_j>0} \lambda_j = f, 
\qquad & \mbox{ in } \Omega, \\[10pt]
u=g , \qquad & \mbox{ on } \partial \Omega.
\end{array}
\right.
$$
Notice that the requirement on $f$ to be positive is nessary in order to have a solution to the limit
equation.  When one adapts the proofs in the following sections to this case the key fact is that
for $f$ positive the player wants to end the game instead of continue playing for a long time
(since at each move he is paying a running payoff that is strictly positive.
}
\end{remark}

\section{Proofs of the results} \label{sect-proofs}

We begin by stating the usual definition of a viscosity solution to \eqref{1.1}.
Here and in what follows $\Omega$ is a domain in $\R^N$. We refer to
\cite{CIL} for general results on viscosity solutions.

\begin{definition} \label{def.sol.viscosa.1}
A continuous function  $u$  verifies
$$
P_{\lambda, \Lambda}^+ (D^2 u) := \Lambda \sum_{\lambda_j>0} \lambda_j
+\lambda \sum_{\lambda_j<0} \lambda_j  = f
$$
\emph{in the viscosity sense} in $\Omega$ if
\begin{enumerate}
\item for every $\phi\in C^{2}$ such that $ u-\phi $ has a strict
minimum at the point $x \in \overline \Omega$  with $u(x)=
\phi(x)$, we have
$$
P_{\lambda, \Lambda}^+ (D^2 u)(x) := \Lambda \sum_{\lambda_j>0} \lambda_j(D^2 u)(x)
+\lambda \sum_{\lambda_j<0} \lambda_j(D^2 u)(x)   \leq f(x).
$$

\item for every $ \psi \in C^{2}$ such that $ u-\psi $ has a
strict maximum at the point $ x \in \overline{\Omega}$ with $u(x)=
\psi(x)$, we have
$$
P_{\lambda, \Lambda}^+ (D^2 u) := \Lambda \sum_{\lambda_j>0} \lambda_j(D^2 u)(x)
+\lambda \sum_{\lambda_j<0} \lambda_j(D^2 u)(x)   \geq f(x).
$$
\end{enumerate}
\end{definition}

Now, let us describe in detail the game that we are presenting in this manuscript.
After that, we prove Theorems \ref{teo.dpp} and \ref{teo.converge}.

Let $\Omega \subset\R^N$ be a bounded open set and $\eps>0$ a fixed real number.
The game starts with a token placed at an initial position $x_0 \in \Omega$.
At every step, the only player, Player~$\I$,
chooses an orthonormal basis of ${\mathbb{R}^N}$, $v_1,...,v_N$ and then for each $v_i$ he
chooses either  $\mu_i = \sqrt{\lambda}$ or $\mu_i = \sqrt{\Lambda}$.
Then the position of the token
is moved to $x\pm \eps  \mu_i v_i$ with equal probabilities. 
That is, each position is selected with probability $\frac{1}{2N}$.
After the first round, the game continues from $x_1$ according to the same rules.
This procedure yields a possibly infinite sequence of game states
$x_0,x_1,\ldots$ where every $x_k$ is a random variable.
The game ends when the token leaves $\Omega$, at this point the 
token will be in the boundary strip of
width $\alpha=\eps\max\{\sqrt{\lambda},\sqrt{\Lambda}\}$, given by
\[
\begin{split}\Gamma_\alpha=
\{x\in {\mathbb{R}}^N \setminus \Omega \,:\,\dist(x,\partial \Omega )\leq \alpha\}.
\end{split}
\]
We denote by $x_\tau \in \Gamma_\alpha$ the first point in the
sequence of game states that lies in $\Gamma_\alpha$, so that $\tau$
refers to the first time we hit $\Gamma_\alpha$.
The payoff is determined by two given functions:
$g:\R^N\setminus \Omega \to \R$, the \emph{final payoff function},
and
$f:\Omega \to \R$, the \emph{running payoff function}.
We require $g$ to be continuous and bounded and $f$ to be uniformly continuous and also bounded.
When the game ends, the total payoff is given by 
\[
g (x_\tau) -\frac1{2N} \eps^2 \sum_{k=0}^{\tau -1}  f (x_k).
\]
We can think that when the token leaves a point $x_k$, Player~$\I$ must pay $\frac{1}{2N} \eps^2 f (x_k)$ 
to move to the next position and at the end he receives 
$g (x_\tau)$.

A strategy $S_\I$ for Player~$\I$ is a Borel function defined on the
partial histories that gives a orthonormal base and values $\mu_i$ at every step of the game
\[
S_\I{\left(x_0,x_1,\ldots,x_n\right)}=(v_1,\dots,v_N,\mu_1,\dots,\mu_N).
\]

When Player~$\I$ fix his strategy $S_\I$ we can compute the expected outcome as follows:
Given the sequence $x_0,\ldots,x_n$ with $x_k\in\Omega$ the next game position is distributed according to
the probability
\[
\pi_{S_\I}(x_0,\ldots,x_n,{A})=
\frac{1}{2N}
\sum_{v_i,\mu_i}
\delta_{x_n+ \eps  \mu_i v_i}(A)+\delta_{x_n- \eps  \mu_i v_i}(A)
\]
where $(v_1,\dots,v_N,\mu_1,\dots,\mu_N)=S_\I{\left(x_0,x_1,\ldots,x_n\right)}$.
By using the Kolmogorov's extension theorem and the one step transition probabilities, we can build a
probability measure $\mathbb{P}^{x_0}_{S_\I}$ on the
game sequences. The expected payoff, when starting from $x_0$ and
using the strategy $S_\I$, is
\begin{equation}
\label{eq:defi-expectation}
\mathbb{E}_{S_{\I}}^{x_0}\left[g (x_\tau) -\frac1{2N} \eps^2 \sum_{k=0}^{\tau -1}  f (x_k)\right]=
\int_{H^\infty} \left(g (x_\tau) -\frac1{2N} \eps^2 \sum_{k=0}^{\tau -1}  f (x_k) \right) \ud\mathbb{P}^{x_0}_{S_\I}.
\end{equation}

The \emph{value of the game} is given by
\[
u^\eps(x_0)=\sup_{S_\I}\,
\mathbb{E}_{S_{\I}}^{x_0}\left[g (x_\tau) -\frac1{2N} \eps^2 \sum_{k=0}^{\tau -1}  f (x_k)\right].
\]

\begin{lemma}\label{martingale1}
The sequence of random variables 
$$\{ |x_{k}-x_{0}|^{2}-k \, \lambda \, \epsilon^{2}\}_{k\ge 1}$$ is a supermartingale with respect to the natural filtration
$\left\{\mathcal{F}^{x_{0}}_{k}\right\}_{n\ge 1}$. 
\end{lemma}
\begin{proof}
Let us compute 
\[
\begin{array}{rcl}
\displaystyle \mathbb{E}^{x_{0}}_{S_{I}} \left[ |x_{k}-x_{0}|^{2}  \,|\, \mathcal{F}^{x_{0}}_{k-1}\right] (x_{k-1})
& = &  \displaystyle  \frac{1}{2N}\sum_{i=1}^{N}|x_{k-1}-x_{0}+ \eps  \mu_i v_i |^{2}+|x_{k-1}-x_{0}- \eps  \mu_i v_i |^{2} \\
& = & \displaystyle  \frac{1}{N}\sum_{i=1}^{N}|x_{k-1}-x_{0}|^{2}+\eps^{2}\mu_{i}^{2}\\
&=&\displaystyle  |x_{k-1}-x_{0}|^{2}+ \eps^{2}  \frac{1}{N}\sum_{i=1}^{N}\mu_{i}^{2}\\
&\ge & \displaystyle |x_{k-1}-x_{0}|^2+\eps^{2} \lambda.
\end{array}
\]
Therefore, we have
\[
\begin{array}{rcl}
\mathbb{E}^{x_{0}}_{S_{I}} \left[ |x_{k}-x_{0}|^{2}   -k\, \lambda\,  \epsilon^{2} \,|\, \mathcal{F}^{x_{0}}_{k-1}\right] (x_{k-1})
& \ge & \displaystyle  |x_{k-1}-x_{0}|^2+\eps^{2} \lambda- k\, \lambda  \epsilon^{2}\\
& = & \displaystyle |x_{k-1}-x_{0}|^2-(k-1)\, \lambda \eps^{2} ,
\end{array}
\]
as we wanted to show.
\end{proof}

Applying Doob's optional stopping to the  finite stopping times $\tau\wedge n$  and letting 
$n\to\infty$, we obtain  
$$\mathbb{E}^{x_{0}}_{S_{I}} \left[ |x_{\tau}-x_{0}|^{2}   -\tau \, \lambda\,  \epsilon^{2}\right] \ge 0,
$$
\begin{equation}\label{taubound}
\lambda\, \epsilon^{2} \mathbb{E}_{
S_{I}}^{x_{0}}[\tau]\le  \mathbb{E}^{x_{0}}_{S_{I}} \left[ |x_{\tau}-x_{0}|^{2}  \right] \le C(\Omega)<\infty
\end{equation} and the
 process ends almost surely: 
 \begin{equation}\label{processends}
 \mathbb{P}^{x_{0}}_{S_{I}}( \{\tau <\infty\})=1.
 \end{equation}
We conclude that the expectation \eqref{eq:defi-expectation} is well defined.

To see that the game value satisfies the DPP, we can consider $\tilde u^\eps$, a function that satisfies the DPP
\begin{equation*}
\left\{
\begin{array}{ll}
\displaystyle \tilde u^\eps (x) =
-\frac{1}{2N}\eps^2f(x)+
 \frac{1}{2N} \sup_{v_i,\mu_i}
 \sum_{i=1}^N \tilde u^\eps (x + \eps \mu_i v_i) + \tilde u^\eps (x - \eps \mu_i v_i)
 & x \in \Omega, \\[10pt]
\tilde u^\eps (x) = g(x)  & x \not\in \Omega.
\end{array}
\right.
\end{equation*}
The existence of such a function can be seen by Perron's method. In fact,
the operator given by the RHS of the DPP, that is,
$$
u \mapsto  
\displaystyle
-\frac{1}{2N}\eps^2f(x)+
\frac{1}{2N}
\sup_{v_i,\mu_i}
 \sum_{i=1}^N u (x + \eps \mu_i v_i) + \tilde u (x - \eps \mu_i v_i)
$$
 is in the hypotheses of the main result of \cite{QS}.

Recall that 
we want to prove that $\tilde u^\eps= u^\eps$.
Given $\eta>0$ we can consider the strategy $S_\I^0$ for Player~$\I$ 
that at every step almost maximize 
\[
\frac{1}{2N}
\sup_{v_i,\mu_i}
\sum_{i=1}^N \Big[ \tilde u^\eps (x + \eps \mu_i v_i) + \tilde u^\eps (x - \eps \mu_i v_i) \Big],
\]
that is, we take
\[
S_\I^0{\left(x_0,x_1,\ldots,x_n\right)}=(w_1,\dots,w_N,\nu_1,\dots,\nu_N)
\]
such that 
\[
\begin{split}
\frac{1}{2N}
\sum_{i=1}^N \tilde u^\eps (x_n + \eps \nu_i w_i) + \tilde u^\eps (x_n - \eps \nu_i w_i)
\geq
\quad\quad\quad\quad\quad\quad\quad\quad\quad\quad\\
\frac{1}{2N}
\sup_{v_i,\mu_i}
\sum_{i=1}^N \tilde u^\eps (x_n + \eps \mu_i v_i) + \tilde u^\eps (x_n - \eps \mu_i v_i)
-\eta 2^{-(n+1)} .
\end{split}
\]
With this choice of the strategy we have
\[
\begin{split}
&\mathbb{E}_{S_\I^0}^{x_0}[\tilde u^\eps(x_{n+1})-\frac1{2N} \eps^2 \sum_{k=0}^{n}  f (x_k)-\eta 2^{-(n+1)}|\,x_0,\ldots,x_k]
\\
&\geq
\frac{1}{2N}
\sup_{v_i,\mu_i}
\sum_{i=1}^N \tilde u^\eps (x_n + \eps \mu_i v_i) + \tilde u^\eps (x_n - \eps \mu_i v_i) -\frac1{2N} \eps^2 \sum_{k=0}^{n}  f (x_k)
-\eta 2^{-(n+1)} -\eta 2^{-(n+1)} 
\\
&\geq \tilde u^\eps(x_n)-\eta 2^{-n}-\frac1{2N} \eps^2 \sum_{k=0}^{n-1}  f (x_k),
\end{split}
\]
here we have used that the DPP holds at $x_n$ for $\tilde u^\eps$.
That is we have proved that 
\[
M_n=
\tilde u^\eps(x_n)-\eta 2^{-n}-\frac1{2N} \eps^2 \sum_{k=0}^{n-1}  f (x_k)
\]
is a submartingale with respect to the natural filtration.
\par
Next , we compute
\begin{equation*}
\begin{split}
u^\eps(x_0)
&=\sup_{S_\I}\,\mathbb{E}_{S_{\I}}^{x_0}\left[g(x_\tau)-\frac1{2N} \eps^2 \sum_{k=0}^{\tau-1}  f (x_k)\right]\\
&\geq \mathbb{E}_{S^0_\I}^{x_0}\left[g(x_\tau)-\frac1{2N} \eps^2 \sum_{k=0}^{\tau-1}  f (x_k)\right]\\
&\geq \mathbb{E}_{S^0_\I}^{x_0}\left[g(x_\tau)-\frac1{2N} \eps^2 \sum_{k=0}^{\tau-1}  f (x_k)-\eta 2^{-\tau}\right]\\
&\geq \liminf_{n\to\infty}\mathbb{E}_{S^0_\I}^{x_0}[M_{\tau\wedge n}]
= \mathbb{E}_{S^0_\I}^{x_0}[M_0]=\tilde  u^\eps(x_0)-\eta,
\end{split}
\end{equation*}
where $\tau\wedge n=\min(\tau,n)$, where we have used the optional stopping theorem for $M_{n}$.
Since $\eta$ is arbitrary this proves that $u^\eps \geq \tilde u^\eps$.
Analogously, we have
\[
\begin{split}
&\mathbb{E}_{S_\I^0}^{x_0}[\tilde u^\eps(x_{n+1})-\frac1{2N} \eps^2 \sum_{k=0}^{n}  f (x_k)|\,x_0,\ldots,x_k]
\\
& \qquad \leq
\frac{1}{2N}
\sup_{v_i,\mu_i}
\sum_{i=1}^N \tilde u^\eps (x_n + \eps \mu_i v_i) + \tilde u^\eps (x_n - \eps \mu_i v_i)-\frac1{2N} \eps^2 \sum_{k=0}^{n}  f (x_k)
\\
& \qquad \leq \tilde u^\eps(x_n)-\frac1{2N} \eps^2 \sum_{k=0}^{n-1}  f (x_k),
\end{split}
\]
where we have estimated the strategy for Player~$\I$ by the supremum.
Hence,
\[
M_n=
\tilde u^\eps(x_n)-\frac1{2N} \eps^2 \sum_{k=0}^{n-1}  f (x_k),
\]
is a supermartingale.
Now, we have
\begin{equation*}
u^\eps(x_0)
=\sup_{S_\I}\,\mathbb{E}_{S_{\I}}^{x_0}\left[g(x_\tau)-\frac1{2N} \eps^2 \sum_{k=0}^{\tau-1}  f (x_k)\right]
\leq \limsup_{n\to\infty}\mathbb{E}_{S^0_\I}^{x_0}[M_{\tau\wedge n}]
= \mathbb{E}_{S^0_\I}^{x_0}[M_0]=\tilde  u^\eps(x_0).
\end{equation*}
This proves that $u^\eps \leq \tilde u^\eps$.
We have proved that $u^\eps = \tilde u^\eps$, and hence Theorem~\ref{1.1.teo} follows.

Now our aim is to pass to the limit in the values of the game
$$
u^\eps \to u, \qquad \mbox{as } \eps \to 0
$$
and obtain in this limit process a viscosity solution to \eqref{1.1}.

To obtain a convergent subsequence $u^\eps \to u$ we will use the following
Arzela-Ascoli type lemma. For its proof see Lemma~4.2 from \cite{MPRb}.

\begin{lemma}\label{lem.ascoli.arzela} Let $\{u^\eps : \overline{\Omega}
\to \R,\ \eps>0\}$ be a set of functions such that
\begin{enumerate}
\item there exists $C>0$ such that $\abs{u^\eps (x)}<C$ for
    every $\eps >0$ and every $x \in \overline{\Omega}$,
\item \label{cond:2} given $\eta>0$ there are constants
    $r_0$ and $\eps_0$ such that for every $\eps < \eps_0$
    and any $x, y \in \overline{\Omega}$ with $|x - y | < r_0 $
    it holds
$$
|u^\eps (x) - u^\eps (y)| < \eta.
$$
\end{enumerate}
Then, there exists  a uniformly continuous function $u:
\overline{\Omega} \to \R$ and a subsequence still denoted by
$\{u^\eps \}$ such that
\[
\begin{split}
u^{\eps}\to u \qquad\textrm{ uniformly in}\quad\overline{\Omega},
\end{split}
\]
as $\eps\to 0$.
\end{lemma}

So our task now is to show that the family $u^\eps$ satisfies the hypotheses of the previous lemma.
Let us start by showing that the family is uniformly bounded.

\begin{lemma}\label{lem.ascoli.arzela.acot} 
There exists $C>0$ such that $\abs{u^\eps (x)}<C$ for
    every $\eps >0$ and every $x \in \overline{\Omega}$.
\end{lemma}

\begin{proof} 
We consider $R>0$ such that $\Omega\subset B_R(0)$ and set $M_n=|x_n|^2$.
From the bound \eqref{taubound} we get

%

\[
\E_{S_{I}}^{x_{0}} [\tau]\leq \frac{4 R^2 }{\lambda} \frac{1}{\eps^{2}}.
\]

Next, we claim  that 
$$
\min g-\frac{2R^2 \max |f|}{N\lambda} \leq u^\eps (x) \leq \max g+\frac{2R^2 \max |f|}{N\lambda}
$$
for every $x \in \overline{\Omega}$.
In fact,
\begin{equation*}
\begin{split}
u^\eps(x_0)
&=\sup_{S_\I}\,\mathbb{E}_{S_{\I}}^{x_0}\left[g(x_\tau)-\frac1{2N} \eps^2 \sum_{k=0}^{\tau-1}  f (x_k)\right]\\
&\leq\max g+\sup_{S_\I}\,\mathbb{E}_{S_{\I}}^{x_0}\left[-\frac1{2N} \eps^2 \sum_{k=0}^{\tau-1}  f (x_k)\right]\\
&\leq\max g+\frac1{2N} \eps^2 \max |f| \sup_{S_\I}\,\mathbb{E}_{S_{\I}}^{x_0}\left[\tau\right]\\
&\leq\max g+\frac1{2N} \eps^2 \max |f|\frac{4 R^2 }{\lambda} \frac{1}{\eps^{2}}\\
&\leq\max g+\frac{2R^2\max |f|}{N\lambda}.
\end{split}
\end{equation*}
The lower bound can be obtained analogously.
\end{proof}

Next, we prove that the family satisfies condition (2) in Lemma~\ref{lem.ascoli.arzela}.
To this end, we follow ideas from \cite{MPRb}.
First, we prove the following lemma.

\begin{lemma}
\label{lem:stopping-time}
Let us consider the game played in an annular domain $B_R(y)\setminus \ol B_\delta(y)$.
Then, the exit time $\tau^{*}$ of the game starting at $x_{0}$  verifies
\begin{equation}
\label{eq:bound-to-stopping-time}
\begin{split}
\mathbb E^{x_0}(\tau^{*})\leq \frac{C(R/\delta)\dist(\partial B_\delta(y),x_0)+o(1)}{\eps^2},
\end{split}
\end{equation}
for $x_0\in B_R(y)\setminus \ol B_\delta(y)$.
Here $o(1)\to 0$ as $\eps\to 0$ can be taken depending only on $\delta$ and $R$.
\end{lemma}

\begin{proof}
Let us denote
\[
\begin{split}
g_\eps(x)=\mathbb E^x(\tau).
\end{split}
\]
Since $g_\eps$ is invariant under rotations, we know that $g_\eps$ is radial.

If we assume that the  player wants to maximize the expectation for the exit time,
we have that the function $g_\eps$ satisfies a dynamic programming principle
\[
g_\eps(x)=
1+\frac{1}{2N} \sup_{v_i,\mu_i}
 \sum_{i=1}^N \Big[ g_\eps (x + \eps \mu_i v_i) + g_\eps (x - \eps \mu_i v_i) \Big]
\]
by the above assumptions and that the number of steps always increases by one when making a step.
Further, we denote $h_\eps(x)=\eps^2 g_\eps(x)$ and obtain
\[
h_\eps(x)=\eps^2+\frac{1}{2N} \sup_{v_i,\mu_i}
 \sum_{i=1}^N  \Big[ h_\eps (x + \eps \mu_i v_i) + h_\eps (x - \eps \mu_i v_i) \Big].
\]
If we rewrite the equation as
\[
-N= \sup_{v_i,\mu_i}
\sum_{i=1}^N\mu_i^2 \Big[\frac{ h_\eps (x + \eps \mu_i v_i) + h_\eps (x - \eps \mu_i v_i)-2h_\eps(x)}{2\eps^2\mu_i^2}
\Big]
\] 
we obtain a discrete version of the PDE
\[
-N= \sup_{v_i,\mu_i}
\sum_{i=1}^N\mu_i^2 \frac{\partial^2 h}{\partial v_i^2}.
\]

We denote $r=\abs{x-y}$. Since $h$ is radial, it's eigenvalues are $(h)_{rr}$ with multiplicity 1 and $\frac{(h)_r}{r}$ with multiplicity $N-1$. 
We will look for a solution $u(r)$, concave and radially increasing.
It will verify
\[
-N= \lambda u_{rr}+(n-1)\Lambda \frac{u_r}{r}.
\]
It's solution takes the form 
\[
u(r)=-\frac{Nr^2}{2(\Lambda(N-1)+\lambda)} +a \frac{ r^{1-\frac{\Lambda(N-1)}{\lambda}}}{1-\frac{\Lambda(N-1)}{\lambda}}+b,
\]
if $\lambda\neq\Lambda(N-1)$ and 
\[
u(r)=-\frac{Nr^2}{2(\Lambda(N-1)+\lambda)} +a \log(r)+b,
\]
if $\lambda=\Lambda(N-1)$. Here $a$ and $b$ are two constants. 
We consider $u$ defined in $x\in B_{R+\delta}(y)\setminus \ol B_{\delta}(y)$.
We can choose $a$ and $b$ such that $u'(R+\delta)=0$, $u(\delta)=0$ and such that the function is positive.
The resulting $u$ is concave and radially increasing.
In fact, we have
\[
u'(r)=-\frac{Nr}{\Lambda(N-1)+\lambda} +a r^{-\frac{\Lambda(N-1)}{\lambda}}.
\]

Then, from $u'(R+\delta)=0$ we conclude that
\[
u'(r)=\frac{N(R+\delta)}{\Lambda(N-1)+\lambda} \left(-\frac{r}{R+\delta} + \left(\frac{r}{R+\delta}\right)^{-\frac{\Lambda(N-1)}{\lambda}}\right)
\]
which is positive when $r<R+\delta$. This shows that $u$ is increasing with respect to $r$.
Moreover, $u$ is concave as a function of $r$, 
\[
\lambda u_{rr}=-N-(n-1)\Lambda \frac{u_r}{r}<0.
\]

We extend this function as a solution to the same equation to $\ol B_\delta(y)\setminus \ol B_{\delta-\sqrt\Lambda\epsilon}(y)$.
If we consider $\eps\sqrt\Lambda<\delta$, for each $x\in B_R(y)\setminus \ol B_\delta(y)$ we have $x \pm \eps \mu_i v_i\subset B_{R+\delta}(y)\setminus \ol B_{\delta-\sqrt\Lambda\epsilon}(y)$.
Since 
\[
-N= \sup_{v_i,\mu_i}
\sum_{i=1}^N\mu_i^2 \frac{\partial^2 u}{\partial v_i^2},
\]
and $u$ is smooth we get
\[
-N= \sup_{v_i,\mu_i}
\sum_{i=1}^N\mu_i^2 \Big[ \frac{ u (x + \eps \mu_i v_i) + u (x - \eps \mu_i v_i)-2u(x)}{2\eps^2\mu_i^2} \Big]
+o(1)
\] 
and hence
\[
u(x)=\eps^2+\frac{1}{2N} \sup_{v_i,\mu_i}
 \sum_{i=1}^N  \Big[ u (x + \eps \mu_i v_i) + u (x - \eps \mu_i v_i) \Big]+o(\eps^2).
\]
Then, for $\eps$ small enough
\[
u(x)\geq\frac{\eps^2}{2}+\frac{1}{2N} \sup_{v_i,\mu_i}
 \sum_{i=1}^N \Big[ u (x + \eps \mu_i v_i) + u (x - \eps \mu_i v_i) \Big].
\]

Now, we consider $w$ defined as $u$ in $B_{R}(y)$ and zero outside.
Observe that we have
\[
w(x)\geq\frac{\eps^2}{2}+\frac{1}{2N} \sup_{v_i,\mu_i}
 \sum_{i=1}^N \Big[ w (x + \eps \mu_i v_i) + w (x - \eps \mu_i v_i) \Big].
\] 
for every $x\in B_R(y)\setminus \ol B_\delta(y)$.

It follows that
\[
\begin{split}
\mathbb{E}&[w(x_k)+\frac{k}{2}\eps^2|\,x_0,\ldots,x_{k-1}]
\\
&\leq
\frac{1}{2N} \sup_{v_i,\mu_i}
 \sum_{i=1}^N \Big[ w (x_{k-1} + \eps \mu_i v_i) + w (x_{k-1} - \eps \mu_i v_i) \Big]+\frac{k}{2}\eps^2.
\\
&\leq w(x_{k-1})+\frac{k-1}{2}\eps^2,
\end{split}
\]
if $x_{k-1}\in B_{R}(y)\setminus \ol B_\delta(y)$.
Thus $w(x_k)+ \frac{k}{2}\eps^2$ is a supermartingale, and the optional
stopping theorem yields
\begin{equation*}
\begin{split}
\mathbb{E}^{x_0}[w(x_{\tau^{*}\wedge k})+\frac{1}{2}(\tau^{*}\wedge k) \eps^2]\leq w(x_0).
\end{split}
\end{equation*}
For $x_{\tau^{*}}$ outside $B_{R}(y)$ we have $w(x_{\tau^{*}})=0$ and for
$x_{\tau^{*}}\in\ol B_\delta(y)\setminus \ol B_{\delta-\sqrt\Lambda\epsilon}(y)$ we have
\[
-w(x_{\tau^{*}})\leq o(1).
\]
Furthermore, the estimate
\[
\begin{split}
0\leq w(x_0)\leq C(R/\delta) \dist(\partial B_\delta(y),x_0)
\end{split}
\]
holds for the solutions due to the facts that they are concave and 
 verify $w=0$ on $\partial B_\delta(y)$. Thus,
passing to the limit as $k\to\infty$, we
obtain
\[
\begin{split}
\frac{1}{2}\eps^2\mathbb{E}^{x_0}[\tau^{*}]\leq w(x_0)-\mathbb{E}[w(x_{\tau^{*}})]
\leq C(R/\delta)\dist(\partial B_\delta(y),x_0) +o(1).
\end{split}
\]
This completes the proof.
\end{proof}

We are ready to prove that the family $u_\eps$ is asymptotically equicontinuous.

\begin{lemma}\label{lem.ascoli.arzela.asymp} Given $\eta>0$ there are constants
    $r_0$ and $\eps_0$ such that for every $\eps < \eps_0$
    and any $x, y \in \overline{\Omega}$ with $|x - y | < r_0 $
    it holds
$$
|u^\eps (x) - u^\eps (y)| < \eta.
$$
\end{lemma}

\begin{proof} 
Given $\eps_0$
we consider the boundary strip of
width $\alpha_0=\eps_0\sqrt\Lambda$ given by
\[
\begin{split}\Gamma_{\alpha_0}=
\{x\in {\mathbb{R}}^N \setminus \Omega \,:\,\dist(x,\partial \Omega )< \alpha_0\}.
\end{split}
\]
The case $x,y \in \Gamma_{\alpha_0}$ follows from the uniformity continuity of $g$ in $\Gamma_{\alpha_0}$ (which is compact).
Observe that the Lemma only includes those $x,y\in\partial\Omega$. 
Although here we are interested in $x,y \in \Gamma_{\alpha_0}$ since we will consider the game in our arguments and the token may lie outside $\overline{\Omega}$.

For the case $x,y \in \Omega $ we argue as follows.
We fix the strategy $S_\I^0$ for the game starting at $x$.
We consider a virtual game starting at $y$. We use the same random steps as
the game starting at $x$.
Furthermore, the player adopts
the strategy $S_\I^0$ from the game starting at $x$,
that is, when the game position is at $y_k$ the player makes the choices that would have taken at $x_k=y_k-y+x$ for the game starting at $x$.
At every time we have $|x_k-y_k| =|x-y|$.
We proceed in this way until for the first time $x_n \in \Gamma_{\alpha_0}$ or $y_n \in \Gamma_{\alpha_0}$.
We can separate the payoff in the amount paid by the player up to this moment and the payoff for the rest of the game.
The difference in the payoff for the rest of the game can be bounded with the desired estimate in the case 
$x_n \in \Omega$, $y_n\in \Gamma_{\alpha_0}$ or for $x_n, y_n \in \Gamma_{\alpha_0}$.
The difference for the amount paid by the player before $x_n$ or $y_n$ reaches $\Gamma_{\alpha_0}$, that is,
$\frac1{2N} \eps^2 \sum_{k=0}^{n-1}  f (x_k)-f(y_k)$ can be bounded considering the bound for the expected exit time obtained in the proof of Lemma~\ref{lem.ascoli.arzela.acot} and the fact that $f$ is uniformly continuous.
In fact we have,
$$
\E \left( \left|\frac1{2N} \eps^2 \sum_{k=0}^{n-1}  f (x_k)-f(y_k) \right| \right) \leq \frac{R^2}{N\lambda} \omega_f (|x-y|)
$$
where $\omega_f$ stands for the uniform modulus of continuity of $f$.

Thus, we can concentrate on the case $x\in \Omega$ and $y \in \Gamma_{\alpha_0}$.
By the exterior sphere condition,
there exists $B_\delta(z)\subset \R^N\setminus \Omega$ such that $y\in
\partial B_\delta(z)$.
For a small $\delta$ we know that the difference $|g(y)-g(z)|$ is small, it remains to prove that the difference $|u_\eps(x)-g(z)|$ is small. 

We take  $\eps_0$ such that $\alpha_0<\frac{\delta}{2}$.
Then
$$
M_k = |x_k - z| - \frac{\Lambda}{\delta} \eps^2 k.
$$
is a supermartingale. Indeed,
\[
\begin{split}
\mathbb{E}^{x_0}_{S_\I}[|x_k - z |\,|\,x_0,\ldots,x_{k-1}]
&\leq \max_{||v||=1,\alpha\in\{\sqrt\lambda,\sqrt\Lambda\}} \frac{|x_{k-1} - z+\eps v\alpha|+|x_{k-1} - z -\eps v\lambda|}{2} \\
&\leq \max_{||v||=1,\alpha\in\{\sqrt\lambda,\sqrt\Lambda\}} \sqrt{\frac{|x_{k-1} - z+\eps v\alpha|^2+|x_{k-1} - z -\eps v\lambda|^2}{2}} \\
&\leq \max_{||v||=1,\alpha\in\{\sqrt\lambda,\sqrt\Lambda\}} \sqrt{|x_{k-1} - z|^2+|\eps v\alpha|^2} \\
&\leq |x_{k-1} - z|+\frac{\eps^2\Lambda}{2|x_{k-1}-z|} \\
&\leq  |x_{k-1} - z |+ \frac{\Lambda\eps^2}{\delta}.
\end{split}
 \]
The last inequality holds because $|x_{k-1}-z|>\frac{\delta}{2}$.

Then, the optional stopping theorem implies that
\[
\mathbb{E}^{x}_{S_\I}[|x_{\tau} - z|] \leq |x - z| +
\frac{\Lambda}{\delta}\eps^2 \mathbb{E}^{x}_{S_\I}[\tau].
\]

Next we estimate $\mathbb{E}^{x}_{S_\I}[\tau]$ by the stopping time in Lemma~\ref{lem:stopping-time}, for $R$ such that $\Omega\subset B_R\setminus \ol B_\delta(z)$. In fact, if we play our game in $B_R\setminus \ol B_\delta(z)$
we can reproduce the same movements until we exit
$\Omega$ (this happens before the game in $B_R\setminus \ol B_\delta(z)$ ends since
$\Omega\subset B_R\setminus \ol B_\delta(z)$)
 and hence we get that, for any strategy $S_I$, it holds that
$$\mathbb{E}^{x}_{S_\I}[\tau] \leq \mathbb{E}^{x}_{S_\I}[\tau^*].$$
Note that any strategy in the domain $\Omega$ can be extended to a strategy to the larger ring domain.
Thus, it follows from  \eqref{eq:bound-to-stopping-time} that 
\[
\mathbb{E}^{x}_{S_\I}[|x_{\tau} - z|] \leq |x - z| +
\frac{\Lambda}{\delta} \Big(C(R/\delta)\dist(\partial B_\delta(z),x_0)+o(1)\Big).
\]
Since $\dist(\partial B_\delta(z),x_0)\leq |x-y|$ and $|x - z|\leq |x-y|+\delta$, we have
\[
\mathbb{E}^{x}_{S_\I}[|x_{\tau} - z|] \leq 
\delta +\left(1+\frac{\Lambda}{\delta} C(R/\delta)\right)|x-y|+o(1).
\]
Thus, we have obtained bounds for $\mathbb{E}^{x}_{S_\I}[|x_{\tau} - z|]$ and $\eps^2 \mathbb{E}^{x}_{S_\I}[\tau]$.
Using  these bounds and the fact that $g$ is uniformly continuous, we have
\[
\begin{split}
&\left|\mathbb{E}_{S_{\I}}^{x_0}\left[g (x_\tau) -\frac1{2N} \eps^2 \sum_{k=0}^{\tau -1}  f (x_k)\right]-g(z)\right|\\
&\qquad\qquad\leq\mathbb{E}_{S_{\I}}^{x_0}\left[|g (x_\tau)-g(z)|\right]+\frac1{2N}\eps^2\mathbb{E}^{x}_{S_\I}[\tau] \|f\|_\infty\\
&\qquad\qquad\leq\mathbb{E}_{S_{\I}}^{x_0}\left[|g (x_\tau)-g(z)|\right]+\frac1{2N}\eps^2\mathbb{E}^{x}_{S_\I}[\tau] \|f\|_\infty\\
&\qquad\qquad\leq\PP^{x}_{S_\I}(|x_{\tau} - z|\geq\theta)2\|g\|_\infty + \sup_{\|x-y\|<\theta} \|g(x)-g(y)\|+\frac1{2N}\eps^2\mathbb{E}^{x}_{S_\I}[\tau] \|f\|_\infty\\
&\qquad\qquad\leq\mathbb{E}^{x}_{S_\I}[|x_{\tau} - z|]\frac{2}{\theta}\|g\|_\infty + \sup_{\|x-y\|<\theta} \|g(x)-g(y)\|+\frac1{2N}\eps^2\mathbb{E}^{x}_{S_\I}[\tau] \|f\|_\infty<\eta
\end{split}
\]
for $\theta$, $\delta$, $\eps_0$ and $r_0$ small enough.
Since this holds for every $S_\I$, we have obtained 
\[
|u_\eps(x)-g(z)|<\eta,
\]
completing the proof.
\end{proof}

From Lemma \ref{lem.ascoli.arzela.acot} and Lemma \ref{lem.ascoli.arzela.asymp}
we have that the hypotheses of the Arzela-Ascoli type lemma, Lemma \ref{lem.ascoli.arzela},
are satisfied. Hence we have obtained uniform convergence of $u^\eps$ along a subsequence.

\begin{corollary} There exists a sequence $\eps_j \to 0$ and $u$, a continuous function in $\overline{\Omega}$, such that
$$u^\eps \to u$$
uniformly in $\overline{\Omega}$.
\end{corollary}

Our next task is to show that this limit is a solution to \eqref{1.1.teo}. This ends the proof of Theorem
\ref{teo.converge}. Notice that, since we have uniqueness of solutions to the limit problem we get
convergence of the whole family $\{ u^\eps \}$ as $\eps \to 0$.

\begin{theorem} \label{teo.sol.viscosa}
Let $u^\eps$ be the values of the game. Then, the uniform limit of $u^\eps$, $u$, is the unique viscosity
solution to
\begin{equation}\label{1.1.teo.sec}
\left\{
\begin{array}{ll}
\displaystyle
P_{\lambda, \Lambda}^+ (D^2 u) := \Lambda \sum_{\lambda_j>0} \lambda_j
+\lambda \sum_{\lambda_j<0} \lambda_j  = f , \qquad & \mbox{ in } \Omega, \\[10pt]
u=g , \qquad & \mbox{ on } \partial \Omega.
\end{array}
\right.
\end{equation}
\end{theorem}

\begin{proof}
First, we observe that since $u^\eps =g$ on $\partial \Omega$ we obtain,
form the uniform convergence, that $u =g$ on $\partial \Omega$.

To check that $u$ is a viscosity solution to $P_{\lambda, \Lambda}^+ (D^2 u)  = f $ in $\Omega$, in the
sense of Definition \ref{def.sol.viscosa.1}, let $\phi\in C^{2}$ be such that $ u-\phi $ has a strict
minimum at the point $x \in \Omega$  with $u(x)=
\phi(x)$. We have to check that
$$
P_{\lambda, \Lambda}^+ (D^2 \phi (x) )  -  f (x) \leq 0.
$$
As $u^\eps \to u$ uniformly in $\overline{\Omega}$ we have the existence of a sequence
$x_\eps$ such that $x_\eps \to x$ as $\eps \to 0$ and 
$$
u^\eps (z) - \phi (z) \geq u^\eps (x_\eps) - \phi (x_\eps) - \eps^3
$$
(remark that $u^\epsilon$ is not continuous in general). 
Since $u^\eps$ is a solution to the DPP
$$
u^\eps (x) =
-\frac{1}{2N}\eps^2f(x)
+ \frac{1}{2N} \sup_{v_i,\mu_i}
\sum_{i=1}^N u^\eps (x + \eps \mu_i v_i) + u^\eps (x - \eps \mu_i v_i)
$$
we obtain that $\phi$ verifies the inequality
$$
0 \geq 
-\frac{1}{2N}\eps^2f(x_\eps)
+ \frac{1}{2N} \sup_{v_i,\mu_i}
\sum_{i=1}^N \Big[ \phi (x_\eps + \eps \mu_i v_i) + \phi (x_\eps - \eps \mu_i v_i) - \phi (x_\eps) \Big] - \eps^3.
$$

Now, consider the Taylor
expansion of the second order of $\phi$
\[
\phi(y)=\phi(x)+\nabla\phi(x)\cdot(y-x)
+\frac12\langle D^2\phi(x)(y-x),(y-x)\rangle+o(|y-x|^2)
\]
as $|y-x|\rightarrow 0$. Hence, we have
\begin{equation} \label{22.q}
\phi(x+\eps \mu v)=\phi(x)+\eps \mu \nabla\phi(x)\cdot v
+\eps^2 \frac12 \mu^2 \langle D^2\phi(x)v,v\rangle+o(\eps^2)
\end{equation}
and
\begin{equation} \label{33.q}
\phi(x- \eps \mu v)=\phi(x) - \eps \mu \nabla\phi(x)\cdot v
+\eps^2 \frac12 \mu^2 \langle D^2\phi(x)v,v\rangle+o(\epsilon^2).
\end{equation}
Hence, using these expansions we get
$$
\frac{1}{2} \phi (x_\eps + \eps \mu v) + \frac{1}{2} \phi (x_\eps - \eps \mu v)
- \phi (x_\eps) = \frac{\eps^2}2 \mu^2 \langle D^2 \phi (x_\eps)v, v \rangle + o(\eps^2),
$$
and then we conclude that
$$
0 \geq -\frac{1}{2N}\eps^2f(x_\eps) + \frac{\eps^2}{2} 
 \frac{1}{N} \sup_{v_i,\mu_i}
\sum_{i=1}^N \Big[\mu_i^2 \langle D^2 \phi (x_\eps)v_i, v_i \rangle \Big] + o(\eps^2).
$$
Dividing by $\eps^2/2N$ and passing to the limit as $\eps \to 0$ we get
$$
0 \geq - f(x) +  
\sup_{v_i,\mu_i}
\sum_{i=1}^N \Big[\mu_i^2 \langle D^2 \phi (x)v_i, v_i \rangle \Big],
$$
that is,
$$
\begin{array}{l}
\displaystyle
0 \geq - f(x) +  
 \sup_{v_i,\mu_i}
\sum_{i=1}^N \Big[\mu_i^2 \langle D^2 \phi (x)v_i, v_i \rangle \Big]
\\[10pt]
\displaystyle
\qquad = - f(x) +  
  \sup_{v_i} \left[
  \Lambda \sum_{\langle D^2 \phi (x)v_i, v_i \rangle>0} \langle D^2 \phi (x)v_i, v_i \rangle
+\lambda \sum_{\langle D^2 \phi (x)v_i, v_i \rangle<0} \langle D^2 \phi (x)v_i, v_i \rangle 
\right]\\[10pt]
\displaystyle
\qquad = - f(x) +  
  \left[ \Lambda \sum_{\lambda_j(D^2 \phi (x))>0} \lambda_j (D^2 \phi (x))
+\lambda \sum_{\lambda_j(D^2 \phi (x))<0} \lambda_j (D^2 \phi (x))
\right]
\end{array}
$$
as we wanted to show. 

The reverse inequality when a smooth function $\psi$
touches $u$ from below can be obtained in a similar way.
\end{proof}

\medskip

{\bf Acknowledgements.} 

Partially supported by CONICET grant PIP GI No 11220150100036CO
(Argentina), by  UBACyT grant 20020160100155BA (Argentina) and by MINECO MTM2015-70227-P
(Spain).


\end{document}